\documentclass[10pt,a4paper]{article}

\usepackage{amsmath}
\usepackage[authoryear]{natbib}
\usepackage{microtype}
\usepackage[bitstream-charter]{mathdesign}

\usepackage{oldgerm}
\DeclareFontFamily{U}{yswab}{}	
\DeclareFontShape{U}{yswab}{m}{n}{<-> yswab}{}

\usepackage{bm}
\usepackage[final]{showkeys}

\usepackage{amsfonts}
\usepackage{dsfont}
\usepackage{color}
\usepackage{xcolor}

\usepackage{amsmath}
\usepackage{amsthm}
\usepackage{color}
\usepackage{xcolor}
\theoremstyle{definition}

\newtheorem{theorem}{Theorem}
\newtheorem{proposition}[theorem]{Proposition}%
\newtheorem{lemma}{Lemma}

\newtheorem{corollary}{Corollary}%
\newtheorem{remark}{Remark}%

\numberwithin{equation}{section}

\newcommand{\sd}{\mathrm{d}}
\newcommand{\e}{\mathrm{e}}
\newcommand{\ci}{\mathrm{i}}
\newcommand\xqed[1]{%
  \leavevmode\unskip\penalty9999 \hbox{}\nobreak\hfill
  \quad\hbox{#1}}
\newcommand\triag{\xqed{$\triangle$}}

\title{The Fractional Non-homogeneous Poisson Process}

\author{$\text{Nikolai Leonenko}^1$ \& $\text{Enrico Scalas}^2$ \& $\text{Mailan Trinh}^2$\\
	\footnotesize{${}^1$Cardiff University, School of Mathematics, Senghennydd Road, Cardiff, CF24 4AG, UK}\\
	\footnotesize{${}^2$Department of Mathematics, School of Mathematical and Physical Sciences}\\
	\footnotesize{University of Sussex, Brighton, BN1 9QH, UK}}

\begin{document}

	\maketitle
	
	\begin{abstract}

		We introduce a non-homogeneous fractional Poisson process by replacing the time variable in the fractional Poisson process of renewal type with an appropriate function of time. We characterize the resulting process by deriving its non-local governing equation. We further compute the first and second moments of the process. Eventually, we derive the distribution of arrival times. Constant reference is made to previous known results in the homogeneous case and to how they can be derived from the specialization of the non-homogeneous process.

		\medskip		
		

	\end{abstract}

\section{Introduction}
\label{sec:introduction}

There are several different approaches to the concept of fractional (homogeneous) Poisson process. The {\em renewal} approach consists in generalizing the characterization of the Poisson process as counting process related to the sum of independent and identically distributed (i.i.d.) non-negative random variables, where, instead of assuming that these random variables follow the exponential distribution, one assumes that they have the Mittag-Leffler distribution. More explicitly if $\{ J_i \}_{i=1}^\infty$ is a sequence of i.i.d. non-negative random variables (with the meaning of inter-event durations), one can define the {\em epochs}
\begin{equation}
T_n = \sum_{i=1}^n J_i,
\end{equation}
and the counting process
\begin{equation}
N(t) = \mathrm{max} \{n: T_n \leq t\}.
\end{equation}
In this renewal context, for $\alpha \in (0,1]$, if one chooses
\begin{equation}
F_J (u) = \mathbb{P} \{J \leq u\} = 1 - E_\alpha (-t^{\alpha}), 
\end{equation}
where $E_\alpha (z)$ is the one-parameter Mittag-Leffler function defined as
\begin{equation}
E_\alpha (z) = \sum_{n=0}^\infty \frac{z^n}{\Gamma(\alpha n + 1)},
\end{equation}
one can define the counting process $N (t)$ as a fractional homogeneous Poisson process. This distribution was used in the framework of queuing theory in \cite{Gnedenko_1968}, where it was not explicitly recognized that it is the limiting distribution of thinning iterations starting from a power-law distribution of $\{ J_i \}_{i=1}^\infty$, and in \cite{Khintchine_1969}.
The renewal approach fully characterizes a specific process and was used in \citep{Mainardi_2004}.
The process defined above was studied in \cite{Beghin_Orsingher_2009, Beghin_Orsingher_2010}, \cite{Meerschaert_2011}, \cite{Politi_2011} among the others. In particular, \cite{Beghin_Orsingher_2009, Beghin_Orsingher_2010} developed the renewal approach to FHPP and proved that its one-dimensional distributions coincide with the solution to the fractional Poisson differential-difference equations. 

Indeed, one can show that the counting probabilities $f_x (t) = \mathbb{P}\{ N (t) = x \}$ obey the following governing equation
\begin{equation}
D_t^\alpha f_x(t) =- f_x(t) +  f_{x-1}(t),
\end{equation}
with appropriate initial conditions, where the operator $D_t^\alpha$ is the fractional Caputo-Djrbashian derivative defined below. The governing equation approach was originally developed in \cite{Repin_2000} and \cite{Laskin_2003}. Parameter estimation for the fractional Poisson distribution is studied in \cite{Cahoy_2010}. This approach does not uniquely define a stochastic process, however, as explicitly shown in \cite{Beghin_Orsingher_2009}. In fact, the solutions of the governing equation only give the one-point counting distribution and nothing is said on all the other finite-dimensional distributions.

A third approach to FHPP is using the {\em inverse subordinator} as in \cite{Meerschaert_2011}. It can be shown that a classic Poisson process coincides in law with an FHPP, in which the time variable is replaced by an independent inverse stable subordinator. This result unifies the two main approaches discussed above.

For the sake of completeness, we must mention a further approach to fractional Poisson processes. This consists in replacing the Gaussian measure in the definition of fractional Brownian motion with the Poisson counting measure. This integral representation method was developed by 
\cite{Wang_Wen_2003}, \cite{Wang_2006} and \cite{Wang_2007}. For other aspects of this approach, the reader is referred to
\cite{Bierme_2013} and \cite{Molchanov_2015b,Molchanov_2015a}.

In this paper, we introduce a fractional non-homogeneous Poisson process (FNPP) following the approach of replacing the time variable in a Poisson counting process $N (t)$ with an appropriate function of time $\Lambda (t)$ in order to get the non-homogeneous process $N (\Lambda (t))$ and further replacing time with $Y_\alpha (t)$, the inverse stable subordinator,  as specified in the following. In other words, we discuss a time-transformed non-homogeneous Poisson process.

\section{Definition and marginal distributions}
\label{sec:defin-marg-distr}

Let $L_\alpha = \{L_\alpha(t), t\geq 0\}$, be an $\alpha$-stable subordinator with Laplace exponent $\phi(s) = s^\alpha$, $0<\alpha<1$, $s\geq 0$, that is $\log\left( \mathbb{E}\left[ \exp(-sL_\alpha(t))\right]\right) = -t\phi(s)$.

Then the inverse stable subordinator \citep[see e.g.][]{Bingham_1971}
\begin{equation}
  \label{eq:1}
  Y_\alpha(t) = \inf\{u\geq 0: L_\alpha(u) > t\}
\end{equation}
has density \citep[see e.g.][]{Meerschaert_2013, Leonenko_2015}
\begin{equation}
  \label{eq:2}
  h_\alpha(t,x) = \frac{t}{\alpha x^{1+\frac{1}{\alpha}}} g_\alpha\left( \frac{t}{x^{\frac{1}{\alpha}}}\right), \quad x\geq 0, t\geq 0,
\end{equation}
where the density of $L_\alpha(1)$ is
\begin{align}
  g_\alpha(z) &= \frac{1}{\pi} \sum_{k=1}^{\infty}(-1)^{k+1}\frac{\Gamma(\alpha k+1)}{k!} \frac{1}{z^{\alpha k +1}} \sin(\pi k \alpha)\nonumber\\
  &=\frac{1}{x}W_{-\alpha,0}\left(-\frac{1}{z^\alpha}\right), z>0.
\end{align}
Here, we use Wright's generalized Bessel function
\begin{equation*}
  W_{\gamma,\beta}(z) = \sum_{k=0}^{\infty}\frac{z^k}{\Gamma(1+k)\Gamma(\beta+\gamma k)},\quad  z \in \mathbb{C}, \gamma > -1, b\in \mathbb{R}
\end{equation*}

Let $N=\{N(t), t\geq 0\}$ be a non-homogeneous Poisson process (NPP) with intensity function $\lambda(t): [0,\infty) \longrightarrow [0,\infty)$.  We denote for $0\leq s < t$
\begin{equation*}
  \Lambda(s,t) = \int_s^t\lambda(u) \sd u,
\end{equation*}
where the function $\Lambda(t) = \Lambda(0,t)$ is known as rate function or cumulative rate function. Thus, $N$ is a stochastic process with independent, but not necessarily stationary increments. Let $0 \leq v < t$. Then the Poissonian marginal distributions of $N$
\begin{align}
  p_x(t,v) &= \mathbb{P}\{N(t+v) - N(v) = x\} \nonumber\\
  &=\frac{\e^{-(\Lambda(t+v) - \Lambda(v))} (\Lambda(t+v) - \Lambda(v))^x}{x!}\nonumber\\
  &=\frac{\e^{-\Lambda(v,t+v)} \Lambda(v,t+v)^x}{x!}, \quad x=0,1,2, \ldots, \label{eq:14}
\end{align}
satisfy the following differential-difference equations:
\begin{equation}
  \label{eq:5}
  \frac{\sd}{\sd t} p_x(t,v) = -\lambda(t+v) p_x(t,v) + \lambda(t+v)p_{x-1}(t,v), \quad x=0,1,2, \ldots,
\end{equation}
with initial conditions
\begin{equation}
  \label{eq:6}
  p_x(0,v) = \left\{
    \begin{array}{ll}
      1, &x=0\\
      0, &x\geq 1
    \end{array}
\right.
\end{equation}
and $p_{-1}(t,v) \equiv 0$. For notational convenience define $p_x(t) := p_x(t,0)$.

If $\lambda(t) = \lambda > 0$ is a constant, we obtain the homogeneous Poisson process (HPP) with intensity $\lambda>0$. We denote this process $N_\lambda(t), t\geq 0$. Observe that
\begin{equation}
  \label{eq:7}
  N(t) = N_1(\Lambda(t)), \quad t\geq 0.
\end{equation}
\begin{remark}
  It can be verified that the FHPP belongs to the class of Cox processes (or doubly stochastic Poisson processes or mixed Poisson processes). This can be done by using a results in \citep{Yannaros_1994} which we will state here for readers' convenience:
  \begin{lemma}
\label{sec:defin-marg-distr-2}
    An ordinary renewal process whose interarrival distribution function $F_J$ satisfies
    \begin{equation}
      \label{eq:23}
      F_J (t) = 1 - \int_0^\infty \e^{-tx} \sd V(x),
    \end{equation}
    where $V$ is a proper distribution function with $V(0)=0$ is a Cox process.
  \end{lemma}
The proof of this result uses a lemma due to \citet{Kingman_1964}, which is formulated for the Laplace transform of $F_J$:
\begin{lemma}
\label{sec:defin-marg-distr-1}
  An ordinary renewal process with interarrival distribution function $F_J$ is a Cox process if and only if the Laplace transform $\hat{F}_J$ of $F_J$ satisfies
  \begin{equation}
    \label{eq:25}
    \hat{F}_J (s) = \frac{1}{1 - \ln(\hat{G}(s))},
  \end{equation}
  where $\hat{G}$ is the Laplace transform of an infinitely divisible distribution function $G$.
\end{lemma}
Both lemmata are powerful tools to check whether a renewal process also belongs to the class of Cox processes. Especially, Lemma \ref{sec:defin-marg-distr-1} gives a full characterization of renewal Cox processes via their Laplace transform. In the case of the FHPP the conditions of both lemmata can be verified. We assume $\lambda=1$ in this remark. To this end, as mentioned in the introduction, recall that the interarrival times $J$ of the FHPP can be expressed by the one-parameter Mittag-Leffler function \citep[see][]{Politi_2011}:
\begin{equation*}
  F_J(t) = 1- E_\alpha(-t^\alpha).
\end{equation*}
Moreover, it can be found in \cite{Mainardi_Gorenflo_2000} that
\begin{gather*}
  \int_0^\infty \e^{-rt} K_\alpha(r) \sd r = E_\alpha(-t^\alpha), \quad \text{where } K_\alpha(r) = \frac{1}{\pi}\frac{r^{\alpha-1}\sin(\alpha\pi)}{r^{2\alpha} + 2r^\alpha \cos(\alpha \pi) +1}.
\end{gather*}
For $0<\alpha<1$ the function $K_\alpha(r)$ is positive and qualifies as a probability density as $\int_0^\infty K_\alpha(r) \sd r =1$. Therefore, the function $V(x) := \int_0^x K_\alpha(r) \sd r$ fulfills the conditions of Lemma \ref{sec:defin-marg-distr-2}.

Alternatively, it is also possible to apply Lemma \ref{sec:defin-marg-distr-1}: In \cite{Meerschaert_2011} it is proven that
\begin{equation*}
  \hat{F}_J(s) = \frac{1}{1+s^\alpha} = \frac{1}{1 - \ln(\exp(-s^\alpha))}.
\end{equation*}
As $\exp(-s^\alpha)$ is the characteristic function of the distribution of the $\alpha$-stable subordinator  at time $t=1$, Lemma \ref{sec:defin-marg-distr-1} may be applied. \triag
\end{remark}
We define the FNPP as
\begin{equation}
  \label{eq:8}
  N_{\alpha}(t) = N_1(\Lambda(Y_\alpha(t))), \quad t\geq 0, 0<\alpha<1, 
\end{equation}
where $Y_\alpha$ is the inverse stable subordinator independent of the HPP $N_1$.
It follows that for $\lambda(t)=\lambda > 0$, the FNPP coincides with the FHPP discussed by \cite{Meerschaert_2011}. In this case the marginal probabilities
\begin{align}
  p_x^\alpha(t) &= \mathbb{P}\{N_\lambda(Y_\alpha(t))=x\}\nonumber\\
  &= \int_0^\infty\e^{-\lambda u}\frac{(\lambda u)^x}{x!}h_\alpha(t,u) \sd u = (\lambda t^\alpha)^xE_{\alpha,\alpha x+1}^{x+1}(-\lambda t^\alpha), \label{eq:9}
\end{align}
where the three-parameter generalized Mittag-Leffler function is defined as follows 
\begin{equation*}  
  E_{a,b}^c(z) = \sum_{j=0}^\infty \frac{(c)_j z^j}{j! \Gamma(aj+b)},
\end{equation*}
where $(c)_j = c(c-1)(c-2)\ldots(c-j+1)$ (also known as Pochhammer symbol) and $a>0, b>0, c>0, z \in \mathbb{C}$. This general form was introduced by \citet{Prabhakar_1971}. As special cases we have for $c=1$ the two-parameter Mittag-Leffler function $E_{a,b}$ and for $b=c=1$ the one-parameter Mittag-Leffler function $E_a$ \citep[see for example][]{Haubold_2011}. 

We will use the fractional Caputo-Djrbashian derivative which is defined as 
\begin{equation}
  \label{eq:3}
  D_t^\alpha f(t) = \frac{1}{\Gamma(1-\alpha)} \int_0^t \frac{\sd f(\tau)}{\sd \tau}\frac{\sd \tau}{(t-\tau)^\alpha}, \quad 0<\alpha<1.
\end{equation}
Its Laplace transform is
\begin{equation}
  \label{eq:4}
  \mathcal{L}\{D_t^\alpha f\}(s) = s^\alpha\mathcal{L}\{f\}(s) - s^{\alpha-1}f(0^+),
\end{equation}
Note that the Laplace transform of $h_\alpha(t,x)$, given in (\ref{eq:2}), is of the form
\begin{equation}
  \label{eq:18}
  \tilde{h}_\alpha(s,x) = s^{\alpha-1}\e^{-x s^\alpha},
\end{equation}
Equations (\ref{eq:3}), (\ref{eq:4}) and (\ref{eq:18}) can be found in \citet[p. 34]{Meerschaert_Sikorskii_2012}.

\cite{Beghin_Orsingher_2009, Beghin_Orsingher_2010} showed that the functions given in (\ref{eq:9}) satisfy the following fractional differential-difference equations:
\begin{equation}
  \label{eq:10}
  D_t^\alpha p_x^\alpha(t) = - \lambda(p_x^\alpha(t)-p_{x-1}^\alpha(t)), \quad x=0, 1, 2, \ldots
\end{equation}
with initial condition
\begin{equation*}
  p_x^\alpha (0) = \left\{ 
      \begin{array}{ll}
        1, &x=0\\
        0, &x\geq 1,
      \end{array}
\right.
\quad \text{and } p_{-1}^\alpha(t)\equiv 0.
\end{equation*}
In the next section we will prove a similar result using the FNPP that includes both the NPP and the FHPP as special cases. In particular, we look for a stochastic process whose marginal distributions give rise to a governing equation that generalizes both equations (\ref{eq:5}) and (\ref{eq:10}).\\
To this end, it is useful to consider the stochastic process $\{I(t,v), t\geq 0\}$ for $v\geq 0$ as
\begin{equation*}
  I(t,v) = N_1(\Lambda(t+v)) - N_1(\Lambda(v))
\end{equation*}
to which we will refer as the increment process of the NPP. The fractional increment process of the NPP is given by
\begin{equation}
  \label{eq:17}
  I_\alpha(t,v) := I(Y_\alpha(t), v) = N_1(\Lambda(Y_\alpha(t)+v)) - N_1(\Lambda(v)).
\end{equation}
and its marginals will be denoted as
\begin{align}
  f_x^\alpha(t,v) &:= \mathbb{P}\{N_1(\Lambda(Y_\alpha(t) + v)) - N_1(\Lambda(v)) = x\} , \quad x=0,1,2,\ldots\nonumber\\
  &= \int_0^\infty p_x(u,v)h_\alpha(t,u)\sd u \nonumber\\
  &= \int_0^\infty \frac{\e^{-\Lambda(v,u+v)} \Lambda(v,u+v)^x}{x!}h_\alpha(t,u) \sd u. \label{eq:15}
\end{align}
For the FNPP the marginal distributions are given by
\begin{align}
  f_x^\alpha (t,0) &= \mathbb{P}\{N_\alpha(t) = x\} = \int_0^\infty p_x(u)h_\alpha(t,u)\sd u\nonumber\\
  &= \int_0^\infty \frac{\e^{-\Lambda(u)}\Lambda(u)^x}{x!}h_\alpha(t,u) \sd u, \quad x=0,1,2,\ldots \label{eq:11}
\end{align}
For shorthand notation we write $f_x^\alpha(t) := f_x^\alpha(t,0)$.\\
Incidentally, this model includes Weibull's rate function:
\begin{equation*}
  \Lambda(t)=\left( \frac{t}{b}\right)^c, \lambda(t) = \frac{c}{b}\left( \frac{t}{b}\right)^{c-1}, c\geq 0, b > 0,
\end{equation*}
the Makeham's rate function
\begin{equation*}
  \Lambda(t) = \frac{c}{b}\e^{bt}-\frac{c}{b} + \mu t, \;\lambda(t) = c\e^{bt}+\mu, \quad c>0, b>0, \mu\geq 0
\end{equation*}
and many others.

\begin{remark}
Note that in general the NFPP does not belong to the class of Cox processes $\tilde{N}(t) = N_1(\tilde{\Lambda}(t))$, $t > 0$, where $\tilde{\Lambda}(t) = \int_0^t \tilde{\lambda}(u)$, $u\geq 0$, is a non-negative stochastic process, see \citet[p.169]{Daley_2003}.\\
In particular, the marginal distribution 
\begin{equation*}
  \mathbb{P}(\tilde{N}(t) = x) = \mathbb{E}\left\{ \frac{\xi([0,t])^x \e^{-\xi([0,t])}}{x!} \right\} = \int_0^\infty \frac{u^x}{x!} \e^{-u} F_A(\sd u), \quad x= 0,1, \ldots
\end{equation*}
is different from (\ref{eq:11}), where $F_A$ ($A=[0,1]$) is the distribution function for the random mass $\xi(A)$, for $\xi$ being a random measure on $[0,\infty)$.
\end{remark}

\section{Governing fractional differential-integral-difference equations}
\label{sec:govern-fract-diff}

We are now ready to derive the governing equation for the fractional increment process. This will lead us to derive a governing equation for the
marginal distribution \eqref{eq:11} of the FNPP.

\subsection{Governing equations}

\begin{theorem}
Let $I_\alpha(t,v)$ be the fractional increment process defined in (\ref{eq:17}). Then, its marginal distribution given in (\ref{eq:15}) satisfies the following fractional differential-integral equations
\begin{equation}
\label{eq:16}
  D_t^\alpha f^\alpha_x(t,v) = \int_0^\infty \lambda(u+v) [-p_x(u,v) + p_{x-1}(u,v)] h_\alpha(t,u)\sd u, \quad x= 0,1, \ldots,
\end{equation}
with initial condition
\begin{equation}
  f_x^\alpha(0,v) = \left\{
    \begin{array}{ll}
      1, &x=0,\\
      0, &x\geq 1
    \end{array}
\right.
\end{equation}
and $f_{-1}^\alpha (0,v) \equiv 0$, where $p_x(u,v)$ is given by (\ref{eq:14}) (with $p_{-1}(u,v)=0$) and $h(t,u)$ is given by (\ref{eq:2}).
\end{theorem}

\begin{proof}
  The initial conditions are easily checked using the fact that $Y_\alpha(0) = 0 \text{ a.s}$ and it remains to prove (\ref{eq:16}). Let $f_x^\alpha$ be defined as in Equation~(\ref{eq:15}). Taking the characteristic function of $f_x^\alpha$ and the Laplace transform w.r.t. $t$ yields
  \begin{align*}
    \bar{f}^\alpha_y(r,v) &= \int_0^\infty \hat{p}_y(u,v) \tilde{h}_\alpha(r,u) \sd u\\
    &= \int_0^\infty\exp(\Lambda(v,u+v)(\e^{\ci y} -1))r^{\alpha-1}\e^{-ur^\alpha} \sd u.
  \end{align*}
Using integration by parts we get
\begin{align*}
  &\bar{f}^\alpha_y(r,v) = r^{\alpha-1} \left[-\frac{1}{r^\alpha} \right.\underbrace{\left.\vphantom{\frac{1}{r^\alpha}}\e^{-ur^\alpha}\exp(\Lambda(v,u+v)(\e^{\ci y} -1))\right\vert_{u=0}^\infty}_{=1}\\ 
    &\left. + \frac{1}{r^\alpha}\int_0^\infty \left( \frac{\sd}{\sd u} \Lambda(v,u+v)\right)(\e^{\ci y}-1)\exp(\Lambda(v, u+v)(\e^{\ci y} -1))\e^{-ur^\alpha} \sd u\right]\\
&= \frac{1}{r^\alpha}\left[ r^{\alpha-1} + (\e^{\ci y}-1)\int_0^\infty \lambda(u+v) \right.
\left.\vphantom{\int}\exp(\Lambda(v,u+v)(\e^{\ci y} -1))r^{\alpha-1}\e^{-ur^\alpha} \sd u\right].
\end{align*}
Now we are able to calculate the Caputo-Djrbashian derivative in Laplace space using Equation (\ref{eq:4}). Note that $\bar{f}^\alpha_y(0^+,v) = 1$ as $Y_\alpha(0) = 0$ a.s.
\begin{align*}
  r^\alpha \bar{f}^\alpha_y(r,v) &- r^{\alpha-1} \\
&= (\e^{\ci y}-1)\int_0^\infty\lambda(u+v)\exp(\Lambda(v,u+v)(\e^{\ci y} -1))r^{\alpha-1}\e^{-ur^\alpha} \sd u\\
&=(\e^{\ci y}-1)\int_0^\infty \lambda(u+v) \hat{p}_y(u,v) \tilde{h}_\alpha(r,u)\sd u.
\end{align*}
Inversion of the Laplace transform yields
\begin{equation*}
  D_t^\alpha \hat{f}^\alpha_y(t,v)=(\e^{\ci y}-1)\int_0^\infty \lambda(u+v) \hat{p}_y(u,v) h_\alpha(t,u)\sd u
\end{equation*}
and finally, by inverting the characteristic function, we obtain
\begin{equation}
  D_t^\alpha f^\alpha_x(t,v) = \int_0^\infty \lambda(u+v) [-p_x(u,v) + p_{x-1}(u,v)] h_\alpha(t,u)\sd u.
\end{equation}
which was to be shown.
\end{proof}

\begin{corollary}
\label{sec:govern-fract-diff-1}
  Let $N_\alpha(t)$, $t\geq 0$, $0<\alpha<1$ be a FNPP given by (\ref{eq:8}). Then, its marginal distributions shown in (\ref{eq:11}) satisfy the following fractional differential-integral equations:
  \begin{equation}
    \label{eq:12}
    D_t^\alpha f_x^\alpha(t) = \int_0^\infty \lambda(u) [-p_x(u) + p_{x-1}(u)] h_\alpha (t,u) \sd u,
  \end{equation}
with initial condition
\begin{equation}
  \label{eq:13}
  f_x^\alpha(0) = \left\{
    \begin{array}{ll}
      1, &x=0,\\
      0, &x\geq 1
    \end{array}
\right.
\end{equation}
and $f_{-1}^\alpha(0) \equiv 0$, where $p_x(u)$ is given by (\ref{eq:14}) and $h_\alpha (t,u)$ is given by (\ref{eq:2}).
\end{corollary}
\begin{proof}
  This follows directly from Theorem \ref{sec:govern-fract-diff-1} with $v=0$.
\end{proof}
\subsection{Special cases}
It is useful to consider two special cases of the governing equations derived above, the FHPP and the NPP.
\begin{enumerate}

\item[(i)]
\label{sec:special-cases}
To get back to the FHPP we choose $\lambda(t) = \lambda >0$ as a constant to get
\begin{align}
  D_t^\alpha f_x^\alpha(t) &= \lambda \int_0^\infty [-p_x(u) + p_{x-1}(u)] h_\alpha (t,u) \sd u \nonumber\\
  &=-\lambda f_x^\alpha(t) + \lambda f_{x-1}^\alpha(t)
\end{align}
which is identical with (\ref{eq:10}). Indeed for constant $\lambda$ in (\ref{eq:11}) we get
\begin{equation*}
  f_x^\alpha (t) = \int_0^\infty \frac{\e^{-u\lambda}(\lambda u)^x}{x!} h_\alpha(t,u) \sd u = p_x^\alpha(t),
\end{equation*}
which means that $f_x^\alpha$ coincides with the marginal probabilities of the FHPP.\\

\item[(ii)]

To obtain the case of the NPP we consider $\alpha \to1$ for which we have $\tilde{h}_1(s,u) = \e^{-us}$ and its Laplace inversion is the delta distribution: $\mathcal{L}\{\tilde{h}\}(t,u) = \delta(t-u)$. By substituting this in Equation (\ref{eq:15}) we get
\begin{equation*}
  f_x^1(t,v) = \int_0^\infty p_x(u,v)\delta(t-u) \sd u = p_x(t,v),
\end{equation*}
which means that $f_x^1$ coincides with the marginal probabilities $p_x$ of the NPP.

Moreover, the proof of Theorem \ref{sec:govern-fract-diff-1} is still valid and by substituting the delta distribution in Equation (\ref{eq:16}) we get for $t \geq 0$
\begin{align}
  D_t^1p_x(t,v) = D_t^1f_x^1(t,v) & = & \nonumber \\ 
\int_0^\infty \lambda(u+v) [-p_x(u+v) + p_{x-1}(u,v)]\delta(t-u) \sd u = \nonumber \\
  \lambda(t+v) [-p_x(t,v) + p_{x-1}(t,v)] & &
\end{align}
which coincides with (\ref{eq:5}).

\end{enumerate}

\section{Moments and covariance structure}
\label{sec:moments-covar-struct}
As a further characterization of the FNPP, we now give the first moments of its distribution, namely the expectation, the variance and the covariance.

\subsection{Expectation}
\label{sec:expectation}

In order to compute the expectation, we use the tower property for the conditional expectation to get
\begin{align}
  \mathbb{E}[N(Y_\alpha(t))] &= \mathbb{E}[\mathbb{E}[N(Y_\alpha(t)) \vert Y_\alpha(t)]] = \int_0^\infty \mathbb{E}[N(x)] h_\alpha(t,x) \sd x \nonumber \\
&= \int_0^\infty \Lambda(x) h_\alpha(t,x) \sd x = \int_0^\infty \int_0^x \lambda(\tau)  h_\alpha(t,x) \sd \tau\sd x \nonumber \\
&=\mathbb{E}[\Lambda(Y_\alpha(t))].
\end{align}

\subsection{Variance}
\label{sec:second-moment-vari}
The variance can be computed by means of the law of total variance:
\begin{align}
  \text{Var}[N(Y_\alpha(t))] &= \mathbb{E}[\text{Var}[N(Y_\alpha(t)) \vert Y_\alpha(t)]] + \text{Var}[\mathbb{E}[N(Y_\alpha(t)) \vert Y_\alpha(t)]] \nonumber \\
 &= \mathbb{E}[\Lambda(Y_\alpha(t))] + \text{Var}[\Lambda(Y_\alpha(t))].
\end{align}
\subsection{Higher moments}
\label{sec:higher-moments}
For fixed $t >0$,  the moments of the Poisson distribution with rate $\Lambda(t)$ can be calculated via the derivatives of its characteristic function. However, the most explicit formula for higher moments of the Poisson distribution is given by
\begin{equation*}
  \mathbb{E}[[N(t)]^k] = \sum_{i=1}^k \Lambda(t)^i
\left\{
  \begin{array}{c}
    k\\
    i
  \end{array}
\right\},
\end{equation*}
where $\left\{
  \begin{array}{c}
    k\\
    i
  \end{array}
\right\}$ are the Stirling numbers of second kind:
\begin{equation*}
  \left\{
  \begin{array}{c}
    k\\
    i
  \end{array}
\right\}
= \frac{1}{i!} \sum_{j=0}^i (-1)^{i-j}
  \left(
  \begin{array}{c}
    i\\
    j
  \end{array}
\right)j^k.
\end{equation*}
Note that the second moment is given by 
\begin{equation*}
  \mathbb{E}[[N(t)]^2] = \Lambda(t) + \Lambda(t)^2, 
\end{equation*}
which we will use later for the calculation of the covariance.\\
Thus for the higher moments of the subordinated process we have
\begin{equation}
  \mathbb{E}[[N(Y_\alpha(t))]^k] = \mathbb{E}\left[ \sum_{i=1}^k \Lambda(Y_\alpha(t))^i
\left\{
  \begin{array}{c}
    k\\
    i
  \end{array}
\right\}\right].
 \end{equation}

\subsection{Covariance}
\label{sec:covariance}
Let $s, t \in \mathbb{R}_{+}$ and w.l.o.g. assume $s < t$. Then
\begin{align*}
  \mathbb{E}[N(s)N(t)] &= \mathbb{E}[N(t) - N(s)] \mathbb{E}[N(s)] + \mathbb{E}[N(s)^2]\\
  &= \Lambda(s,t)\Lambda(0,s) + \Lambda(0,s)^2 + \Lambda(0,s)
\end{align*}
and thus
\begin{align*}
  \text{Cov}(N(s), N(t)) &= \mathbb{E}[N(s)N(t)] - \mathbb{E}[N(s)]\mathbb{E}[N(t)]\\
  &= \Lambda(s,t)\Lambda(0,s) + \Lambda(0,s)^2 + \Lambda(0,s) - \Lambda(0,s)\Lambda(0,t)\\
  &= \Lambda(0,s)[\Lambda(s,t) + \underbrace{\Lambda(0,s) - \Lambda(0,t)}_{=-\Lambda(s,t)} + 1] = \Lambda(0,s).
\end{align*}
The same calculation can be done for the case $t < s$. In short, both cases can be summarized in the following way:
\begin{equation}
  \text{Cov}(N(s), N(t)) = \Lambda(0, s\wedge t).
\end{equation}
\begin{proposition}
By the law of total covariance, one finds:
\begin{align}
  \text{Cov}&[N(Y_\alpha(s)), N(Y_\alpha(t))] = \mathbb{E}[\text{Cov}[N(Y_\alpha(s)),N(Y_\alpha(t)) \vert Y_\alpha(s), Y_\alpha(t)]] \nonumber \\
  &+ \text{Cov}[\mathbb{E}[N(Y_\alpha(s))\vert Y_\alpha(s), Y_\alpha(t)], \mathbb{E}[N(Y_\alpha(t))\vert Y_\alpha(s), Y_\alpha(t)]] \nonumber \\
&=\mathbb{E}[\Lambda(0, Y_\alpha(s \wedge t))] + \text{Cov}[\Lambda(Y_\alpha(s)), \Lambda(Y_\alpha(t))]
\end{align}
\end{proposition}
\begin{proof}
  For the first term, we have
  \begin{align*}
    \mathbb{E}&[\text{Cov}[N(Y_\alpha(s)),N(Y_\alpha(t)) \vert Y_\alpha(s), Y_\alpha(t)]] = \mathbb{E}[\mathbb{E}[N(Y_\alpha(s))N(Y_\alpha(t))] \vert Y_\alpha(s), Y_\alpha(t)]\\
    &- \mathbb{E}[N(Y_\alpha(s)) \vert Y_\alpha(s), Y_\alpha(t)]\mathbb{E}[N(Y_\alpha(t)) \vert Y_\alpha(s), Y_\alpha(t)]\\
    &= \int_0^\infty \int_0^\infty \mathbb{E}[N(x)N(y)] p_{(Y_\alpha(s), Y_\alpha(t))}(x,y) \;\sd x\; \sd y\\
    &- \int_0^\infty \int_0^\infty \mathbb{E}[N(x)] \mathbb{E}[N(y)] p_{(Y_\alpha(s), Y_\alpha(t))}(x,y) \;\sd x\; \sd y\\
    &= \int_0^\infty \int_0^\infty \text{Cov}[N(x), N(y)] p_{(Y_\alpha(s), Y_\alpha(t))}(x,y) \;\sd x \;\sd y\\
    &= \mathbb{E}[\Lambda(0, Y_\alpha(s) \wedge Y_\alpha(t))] = \mathbb{E}[\Lambda(Y_\alpha(s \wedge t))].
  \end{align*}
Note that in the last step we have used that $Y_\alpha$ is an increasing process.

For the second term:
\begin{align*}
  \text{Cov}&[\mathbb{E}[N(Y_\alpha(s))\vert Y_\alpha(s), Y_\alpha(t)], \mathbb{E}[N(Y_\alpha(t))\vert Y_\alpha(s), Y_\alpha(t)]]\\
  &=\mathbb{E}[\mathbb{E}[N(Y_\alpha(s)) \vert Y_\alpha(s), Y_\alpha(t)]\mathbb{E}[N(Y_\alpha(t)) \vert Y_\alpha(s), Y_\alpha(t)]]\\
 &- \mathbb{E}[\mathbb{E}[N(Y_\alpha(s)) \vert Y_\alpha(s), Y_\alpha(t)]] \mathbb{E}[\mathbb{E}[N(Y_\alpha(t)) \vert Y_\alpha(s), Y_\alpha(t)]]\\
  &=\int_0^\infty \int_0^\infty \mathbb{E}[N(x)] \mathbb{E}[N(y)] p_{(Y_\alpha(s), Y_\alpha(t))}(x,y) \;\sd x \;\sd y
  - \mathbb{E}[N(Y_\alpha(s))]\mathbb{E}[N(Y_\alpha(t))]\\
  &=\mathbb{E}[\Lambda(Y_\alpha(s))\Lambda(Y_\alpha(t))]- \mathbb{E}[\Lambda(Y_\alpha(s))]\mathbb{E}[\Lambda(Y_\alpha(t))]\\
  &=\text{Cov}[\Lambda(Y_\alpha(s)),\Lambda(Y_\alpha(t))],
\end{align*}
where $p_{(Y_\alpha (s), Y_\alpha (t))} (x,y)$ is the joint density of $Y_\alpha (s)$ and $Y_\alpha (t)$.
\end{proof}
\begin{remark}
The two-point cumulative distribution function of the inverse stable subordinator $Y_\alpha (t)$ can be computed using the fact that \citep[see][]{Leonenko_2013}
\begin{equation}
\mathbb{P}(Y_\alpha (s) > x, Y_\alpha (t) > y) = \int_{v=0}^t \frac{\alpha}{v} y h_\alpha (s,y) \int_{u=0}^{s-v} \frac{\alpha}{u} (x - y) h_\alpha (t,x-y) \, du dv.
\end{equation}

\end{remark}
\begin{remark}
For the homogeneous case $\Lambda(t) = \lambda t$, we get
\begin{equation*}
  \text{Cov}[N(Y_\alpha(s)), N(Y_\alpha(t))] = \lambda \mathbb{E}[Y_\alpha(s \wedge t)] + \lambda^2 \text{Cov}[Y_\alpha(s),Y_\alpha(t)],
\end{equation*}
which is consistent with the results in \cite{Leonenko_2014}.
\end{remark}

\section{Arrival times}
\label{sec:arrival-times}

Let $T_n = \min\{t \in [0,\infty): N_\alpha(t) =n\}$ be the epochs or event arrival times. Then, the following events coincide: $\{T_n \leq t\} = \{N_\alpha(t) \geq n\}$ and
\begin{align*}
  F_{T_n}(t) := \mathbb{P}(T_n \leq t) &= \mathbb{P}(N_\alpha(t) \geq n) = \sum_{x=n}^\infty f_x^\alpha(t)\\
  &= \sum_{x=n}^\infty \int_0^\infty \frac{\e^{-\Lambda(u)}\Lambda(u)^x}{x!} h_\alpha(t,u) \sd u\\
  &= \int_0^\infty h_\alpha(t,u)\sum_{x=n}^\infty \frac{\e^{-\Lambda(u)}\Lambda(u)^x}{x!}\sd u.
\end{align*}
In the last step we are allowed to use Fubini's theorem as the integrand is positive. Further, by integration by parts we get
\begin{align}
  \mathbb{P}&(T_n \leq t) = \left.\int_0^\infty h_\alpha(t,v) \sd v \left( \sum_{x=n}^\infty \frac{\e^{-\Lambda(u)} \Lambda(u)^x}{x!}\right)\right\vert_{u=0}^\infty \label{eq:20}\\
  &- \int_0^\infty \int_0^u h_\alpha(t,v) \sd v \left[ \sum_{x=n}^\infty \frac{-\Lambda'(u) \e^{-\Lambda(u)} \Lambda(u)^x}{x!} + \frac{-\Lambda'(u) \e^{-\Lambda(u)} \Lambda(u)^{(x-1)}}{(x-1)!}\right]\nonumber
\end{align}
As the power expansion of the exponential function is absolutely convergent, we are allowed to interchange limit and sum in (\ref{eq:20}). The limit is finite and we will denote it by $C$:
\begin{equation*}
  C:= \lim_{u\rightarrow \infty}\sum_{x=n}^\infty \frac{\e^{-\Lambda(u)} \Lambda(u)^x}{x!} < \infty.
\end{equation*}
We have $C=0$ if $\Lambda(u) \longrightarrow \infty$ for $u \longrightarrow \infty$. Thus, we obtain
\begin{align}
  \mathbb{P}(T_n \leq t) &= C + \int_0^\infty \left(\int_0^u h_\alpha(t,v) \sd v \right)\lambda(u) \e^{-\Lambda(u)} \left[ \sum_{x=n}^\infty \frac{\Lambda(u)^{x-1}}{(x-1)!} - \frac{\Lambda(u)^x}{x!}\right] \sd u\nonumber\\
  &= C + \int_0^\infty \left(\int_0^u h_\alpha(t,v) \sd v \right)\lambda(u) \e^{-\Lambda(u)}\frac{\Lambda(u)^{n-1}}{(n-1)!}\sd u.
\end{align}
As the power expansion of the exponential function is absolutely convergent, we are allowed to interchange limit and sum in (\ref{eq:20}). This result generalizes the Erlang distribution for the FNPP.

\section{Summary and outlook}

In this paper, we introduced a new stochastic process, the fractional non-homogeneous Poisson process (FNPP) as $N_\alpha (t) = N_1 (\Lambda (Y_\alpha (t)))$ where $N_1(t)$ is the homogeneous Poisson process with $\lambda =1$, $\Lambda (t)$ is the rate function and $Y_\alpha (t)$ is the inverse stable subordinator. This is a straightforward generalization of the non-homogeneous Poisson process (NHPP) $N_1 (\Lambda (t))$ and it reduces to the NHPP in the case $\alpha =1$. In Theorem 1, we have been able to derive a fractional governing equation for the process
$I_\alpha (t,v) = N_1(\Lambda(Y_\alpha(t)+v)) - N_1 (\Lambda (v))$. For $v=0$, this equation gives the fractional governing equation for the
marginal distributions $f_x^\alpha (t,0) = \mathbb{P}\{ N_\alpha (t) = x \}$. The calculations of moments for this process is a straightforward application of the rules for conditional expectations. Finally, it is possible to derive explicit expressions for the distribution of event arrival times.

As usual in these cases, this is not the only possible fractional non-homogeneous Poisson process. For instance, one could think of the process $N_1 (Y_\alpha (\Lambda (t)))$, where one begins from the fractional homogeneous Poisson process $\widetilde{N}_\alpha (t) = N_1 (Y_\alpha (t))$ and replaces the time with a rate function $\Lambda (t)$. The two processes $N_\alpha (t)$ and $\widetilde{N}_\alpha (t)$ do not coincide and they have different governing equations.

From a heuristic point of view, we expect that non-homogeneous fractional Poisson processes can be useful for modeling systems in which anomalous waiting times do not have stationary distributions. In these cases, it should be possible to use appropriate constructions such as those described in \cite{Gergely_1973} to derive the appropriate mesoscopic or macroscopic process from the microscopic stochastic dynamics. All this will be the subject of further research.

\section*{Acknowledgements}
\label{sec:acknowledgement}

N. Leonenko  was supported in particular by Cardiff Incoming Visiting Fellowship Scheme and International Collaboration Seedcorn Fund and Australian Research Council's Discovery Projects funding scheme (project number DP160101366). E. Scalas and M. Trinh were supported by the Strategic Development Fund of the University of Sussex.

  \bibliographystyle{elsarticle-harv} 
  \bibliography{fractionalRef}






\end{document}